\documentclass[english]{article}
\usepackage[T1]{fontenc}
\usepackage[latin9]{inputenc}
\usepackage{geometry}
\usepackage{amsmath}
\usepackage{amssymb}
\usepackage[authoryear]{natbib}

\makeatletter
\usepackage{amsfonts}
\usepackage{graphicx}

\usepackage{amsfonts}
\usepackage{graphicx}
\usepackage{amsthm}

\theoremstyle{definition}
\newtheorem{theorem}{Theorem}
\newtheorem{corollary}{Corollary} 
\newtheorem{lemma}{Lemma}
\newtheorem{proposition}{Proposition}
\newtheorem{assumption}{Assumption}
\newtheorem{definition}{Definition} 
\theoremstyle{remark}
\newtheorem{example}{Example} 
\newtheorem{remark}{Remark}

\makeatother

\usepackage{babel}
\begin{document}

\title{Brownian semistationary processes and related processes}

\author{Orimar Sauri \\
Department of Mathematics and CREATES\\
 Aarhus University\\
 osauri@math.au.dk }

\date{\today }

\maketitle
 
\begin{abstract}
In this paper we find a pathwise decomposition of a certain class
of Brownian semistationary processes ($\mathcal{BSS}$) in terms of
fractional Brownian motions. To do this, we specialize in the case
when the kernel of the $\mathcal{BSS}$ is given by $\varphi_{\alpha}\left(x\right)=L\left(x\right)x^{\alpha}$
with $\alpha\in(-1/2,0)\cup(0,1/2)$ and $L$ a continuous function
slowly varying at zero. We use this decomposition to study some path
properties and derive Itô's formula for this subclass of $\mathcal{BSS}$
processes.
\end{abstract}


\section{Introduction}

During recent years, \textit{Brownian semistationary processes} ($\mathcal{BSS}$)
and their tempo-spatial analogues, \textit{Ambit fields}, have been
widely studied in the literature. $\mathcal{BSS}$ processes were
originally introduced in \cite{BNSch09} as a model for the temporal
component of the velocity field of a turbulent fluid. They form a
class of processes analogous to that of Brownian semimartingales in
the stationary context. Specifically, $\mathcal{BSS}$ constitute
the family of stochastic processes which can be written as
\[
Y_{t}=\theta+\int_{-\infty}^{t}g\left(t-s\right)\sigma_{s}dB_{s}+\int_{-\infty}^{t}q\left(t-s\right)a_{s}ds,\text{ \ \ }t\in\mathbb{R}\text{,}
\]
where $\theta\in\mathbb{R}$, $B$ is a Brownian motion, $g$ and
$q$ are deterministic functions such that $g\left(x\right)=q\left(x\right)=0$
for $x\leq0$, and $\sigma$ and $a$ are predictable processes. A
very remarkable property is that $\mathcal{BSS}$\ are not semimartingales
in general. A natural extension, the so-called \textit{Lévy semistationary
processes} ($\mathcal{LSS}$), is obtained by replacing $B$ by a
Lévy process. Note that the class $\mathcal{LSS}$ has been used as
models for energy spot prices in \cite{BNBEnthVeraat13}, \cite{BenthEyjolVeraart14},
\cite{Veraart14} and \cite{Benedsen15}.

We would like to emphasize that any $\mathcal{LSS}$ is a null-space
\textit{Ambit field}: for every $t\in\mathbb{R}$ and $x\in\mathbb{R}^{d}$
an Ambit field follows the equation 
\[
Y_{t}\left(x\right):=\theta+\int_{A_{t}\left(x\right)}g\left(t,x;y,s\right)\sigma_{s}\left(y\right)Z\left(dy,ds\right)+\int_{D_{t}\left(x\right)}q\left(t,x;y,s\right)a_{s}\left(y\right)dyds,
\]
where $Z$ is a Lévy basis on $\mathbb{R}^{d+1}$, $g$ and $q$ are
deterministic functions and $\sigma$ and $a$ are suitable stochastic
fields for which the integral exists. Furthermore, the \textit{Ambit
sets} satisfy that $A_{t}\left(x\right),D_{t}\left(x\right)\subset\left(-\infty,t\right]\times\mathbb{R}^{d}$.
For surveys related to the theory and applications of Ambit fields, we refer to \cite{Podolskij15}, \cite{BNBEnthVeraat15} and \cite{BNHedSchSzoz16}. See also \cite{BNBEnthVeraat11}, \cite{BNBEnthVeraat14}
and \cite{Pakkanen14}.

In this paper we consider the subclass of $\mathcal{BSS}$ given
by
\begin{equation}
Y_{t}:=\int_{-\infty}^{t}\varphi_{\alpha}\left(t-s\right)\sigma_{s}dB_{s},\text{ \ \ }t\in\mathbb{R}\text{,}\label{eqnchap11.1}
\end{equation}
where 
\begin{equation}
\varphi_{\alpha}\left(x\right):=L\left(x\right)x^{\alpha},\text{ \ \ }x>0,\label{eqnchap11.2}
\end{equation}
$\alpha\in(-1/2,0)\cup(0,1/2)$ and $\sigma$ is an adapted stationary
càglàd process. The function $L$ is continuous and slowly varying
at $0$, where the latter means that for every $t>0,$ $\lim_{x\downarrow0}L\left(tx\right)/L\left(x\right)=1$.

Kernels of the type of (\ref{eqnchap11.2}) are building blocks for
the standard class of $\mathcal{BSS}$/$\mathcal{LSS}$ considered
in the literature. For example, most of the existing limit theory
for $\mathcal{BSS}$ and $\mathcal{LSS}$ processes is developed under
this framework. See \cite{CorcueraHedPakkPod13} and \cite{BasseHeinrPod17}.
We refer also to \cite{BNCorcPodols11}, \cite{BNCorcPodol13}, \cite{GartPod15}
and \cite{BasseLachPod17}. $\mathcal{BSS}$ processes of the type
of (\ref{eqnchap11.1}) have also proved to be a potential model for
turbulent time series as it was shown in \cite{MarquezSch16}, cf.
\cite{BNSch09}. Furthermore, \cite{BenLundPakk16} have shown that
processes of the kind of (\ref{eqnchap11.1}) can be used as a parsimonious
way of decoupling the short and long term behavior of time series
in terms of $\alpha$ and $L$, respectively 

It was noticed in \cite{BN12}, cf. \cite{CorcueraHedPakkPod13},
that in the case when $L\left(x\right)=e^{-\lambda x},$ the second
order structure of the increments of $Y$ behaves as that of a fractional
Brownian motion with Hurst parameter $H=\alpha+1/2$. A generalization in this direction
that includes certain subclass of functions of the form of (\ref{eqnchap11.2})
can be found in \cite{BenLundPakk16A}. Such a behaviour means that,
in the second order sense, the fractional Brownian motion (fBm for
short) only differs from this kind of $\mathcal{BSS}$ by a process
of bounded variation. The main goal of this paper is to show that,
up to a modification, $Y$ only differs from a fBm by an absolutely
continuous process. To do this, we study a Wiener-type stochastic
integral with respect to \textit{volatility modulated Volterra processes
on the real line (}$\mathcal{VMVP}$)\textit{.} Let us remark that
this class of integrals can be seen as a particular case of those
introduced in \citet{AlosMazerNualart01}, \cite{BNBEnthPedVeraat14}
and \cite{Mocioalca04}, for the case of non-random integrands. Finally,
as a way to show the potential of our main result, we derive some
path properties of $Y$ and, based on the existing literature of stochastic
calculus for the fBm, we establish Itô's formulae for $Y$.

The paper is organized as follows: Section 2 introduces the basic
notation and definitions. Our basic examples are also introduced in that section and we finish by giving some basic properties of fractional
Brownian motions. In Section 3, we introduce the class of $\mathcal{VMVP}$
on the real line and, by similar heuristic arguments as in \cite{BNBEnthPedVeraat14},
cf. \cite{Mocioalca04}, we define Wiener-type stochastic integrals.
We show that in the case of bounded variation integrands, such integral
operator is just of the Lebesgue-Stieltjes type. By using this, we
study the class of solutions of the associated Langevin equation.
Section 4 provides a pathwise decomposition of $\mathcal{BSS}$ processes
admitting the representation of (\ref{eqnchap11.1}), as a sum of
an absolutely continuous process and the so-called volatility modulated
fractional Brownian motion. We end the section by establishing Itô's
formulae in the case when $1/2>\alpha>0$. In Section 5 we resume
our findings. We have also added an appendix which includes some technical
results needed in the paper.

\section{Preliminaries and basic results}

Throughout this paper $\left(\Omega,\mathcal{F},\left(\mathcal{F}_{t}\right)_{t\in\mathbb{R}},\mathbb{P}\right)$
denotes a filtered probability space satisfying the usual conditions
of right-continuity and completeness. We will further assume that
$\left(\Omega,\mathcal{F},\left(\mathcal{F}_{t}\right)_{t\in\mathbb{R}},\mathbb{P}\right)$
is rich enough to support a two-sided Brownian motion $\left(B_{t}\right)_{t\in\mathbb{R}}$,
that is, $B$ is a continuous centered Gaussian process satisfying
that $B_{0}=0$ and for any $t\geq0$ and $s\in\mathbb{R}$, the random
variable $B_{t+s}-B_{s}$ has variance $t$, is $\mathcal{F}_{t+s}$-measurable
and is independent of $\mathcal{F}_{s}.$ A stochastic process $(Z_{t})_{t\in\mathbb{R}}$
is said to be an\textit{ increment semimartingale} on $\left(\Omega,\mathcal{F},\left(\mathcal{F}_{t}\right)_{t\in\mathbb{R}},\mathbb{P}\right)$,
if the process $\left(Z_{t+s}-Z_{s}\right)_{t\geq0}$ is a semimartingale
on $\left(\mathcal{F}_{t+s}\right)_{t\geq0}$ in the usual sense,
for any $s\in\mathbb{R}$. For a detailed study of increment semimartingales,
included a stochastic integration theory, the reader can see \citet{BasseGravPed13},
cf. \citet{BasseGravPed10}. 

\subsection{Existence of $\mathcal{BSS}$ processes and its stationary structure}

In this part we briefly discuss necessary and sufficient conditions
for which the stochastic integral in (\ref{eqnchap11.1}) exists in
the Itô's sense.

Observe that, for any predictable process $\phi$, the stochastic
integral $\int_{\mathbb{R}}\phi_{s}dB_{s}$ exists if and only if
almost surely $\int_{\mathbb{R}}\phi_{s}^{2}ds<\infty$, see for instance
the paper of \cite{BasseGravPed13}. Therefore, the process $Y$ in
(\ref{eqnchap11.1}) is well defined if and only if for every $t\in\mathbb{R}$,
a.s. $\int_{0}^{\infty}\varphi_{\alpha}^{2}\left(s\right)\sigma_{s+t}^{2}ds<\infty$.
Since $\sigma$ is assumed to be càglàd, then its paths are bounded
in compact sets almost surely. Thus, there exists a positive (random)
constant $M_{t}>0$ such that with probability $1$ it holds that
\[
\int_{0}^{1}\varphi_{\alpha}^{2}\left(s\right)\sigma_{s+t}^{2}ds\infty\leq M_{t}\int_{0}^{1}\varphi_{\alpha}^{2}\left(s\right)ds<\infty.
\]
Hence, $Y$ is well defined if and only if $\int_{1}^{\infty}\varphi_{\alpha}^{2}\left(s\right)\sigma_{s+t}^{2}ds<\infty$
almost surely. Thus if $\sigma$ is bounded in $\mathcal{L}^{2}\left(\Omega,\mathcal{F},\mathbb{P}\right)$,
that is $\sup_{t}\mathbb{E}(\sigma_{t}^{2})<\infty$, then $Y$ is
well defined if $\int_{1}^{\infty}\varphi_{\alpha}^{2}\left(s\right)ds<\infty$.
For general conditions on the existence of $Y$ in the non-Gaussian
case, see \citet{Basse13} and \citet{PedSau15}.

It is well known that if $\sigma$ and $B$ are independent and $\sigma$
is weakly (strongly) stationary, then $Y$ is weakly (strongly) stationary
as well. Furthermore, its autocovariance function is given by 
\begin{equation}
\gamma\left(h\right)=\mathbb{E}\left(\sigma_{0}^{2}\right)\int_{0}^{\infty}\varphi_{\alpha}\left(s\right)\varphi_{\alpha}\left(s+h\right)ds.\label{eq:autocovBSS}
\end{equation}
In applications, $\sigma$ might have a certain dependence structure
on $B$, so the imposed independence assumption does not seem to be
general enough. Nevertheless, as the following result shows, we can
drop the independence among $\sigma$ and $B$ but the price to pay
is that we cannot have a closed formula for its autocovariance structure
anymore. For a proof see the Appendix.

\begin{proposition} \label{stationarity}Let $\phi$ be a square
integrable function. Assume that for all $h\in\mathbb{R}$ 
\begin{equation}
\left(\sigma_{s+h},B_{s}^{h}\right)_{t\in\mathbb{R}}\overset{d}{=}\left(\sigma_{s},B_{s}\right)_{t\in\mathbb{R}},\label{eqnchap13.1}
\end{equation}
where 
\[
B_{t}^{h}:=B_{t+h}-B_{h}\text{, \ \ }t\in\mathbb{R}\text{,}
\]
and $\sigma$ is a stationary càglàd adapted process. Then the $\mathcal{BSS}$
given by
\[
Y_{t}:=\int_{-\infty}^{t}\phi\left(t-s\right)\sigma_{s}dB_{s}\text{, \ 		\ }t\in\mathbb{R}\text{,}
\]
is stationary in the strong sense. \end{proposition}

\subsection{Basic examples}

In this part we introduce two important subclasses of $\mathcal{BSS}$ of the type of (\ref{eqnchap11.1}).
\subsubsection*{$\mathcal{BSS}$ with a gamma kernel}

Consider the case when $\varphi_{\alpha}$ can be expressed as a gamma
density, this is
\begin{equation}
\varphi_{\alpha}\left(x\right)=\frac{\lambda^{\alpha+1}}{\Gamma\left(\alpha+1\right)}e^{-\lambda x}x^{\alpha}\text{,}\qquad\alpha>-1/2\text{, }x>0\text{.}\label{8.3.3.0}
\end{equation}
The distinctive property of $\varphi_{\alpha}$ is that it solves,
thou not uniquely, the so-called causal covariance equation 
\begin{equation}
\int_{0}^{\infty}\varphi_{\alpha}\left(x+h\right)\varphi_{\alpha}\left(x\right)\mathrm{d}x=\frac{2^{-\alpha+1/2}}{\Gamma\left(\alpha+1/2\right)}\overline{K}_{\alpha+1/2}\left(\lambda h\right),\qquad h\geq0\text{,}\label{8.3.3.1}
\end{equation}
where $\overline{K}_{\nu}\left(u\right)=u^{\nu}K_{\nu}\left(u\right)$,
with $K_{\nu}$ the modified Bessel function of the second kind with
parameter $\nu$. The function on the right-hand side of (\ref{8.3.3.1})
is known as the {\it Whittle-Mat\'ern} covariance function. For a
historical review on this class of autocovariance functions we refer
to \cite{GuttorpGneiting05} .

It is clear that $\int_{1}^{\infty}\varphi_{\alpha}^{2}\left(s\right)ds<\infty$,
meaning that the associated $\mathcal{BSS}$ process is well defined
and its autocovariance function, in the case when $\sigma$ is independent
of $B$, is proportional to that in (\ref{8.3.3.1}). We will refer
to $Y$ as an $\mathcal{BSS}$ \textit{with a gamma kernel}. This
subclass of Brownian semistationary processes, as we will show later,
are closely related with fractional Ornstein-Uhlenbeck processes.
We would like to emphasize that $Y$ has successfully studied and
applied in the fields of econometrics, finance and turbulence. See
for instance \cite{BNSch09}, \cite{BNCorcPodols11}, \cite{Pakkanen11},
\cite{BNBEnthVeraat13}, and \citet{PedSau15}.

It was shown in \cite{BNSch09} that $Y$ is not in general a semimartingale.
More precisely, for any $\alpha\geq1/2$, $Y$ is a process of bounded
variation. However, when $\alpha\in\left(-1/2,0\right)\cup\left(0,1/2\right),$
$Y$ is no longer a semimartingale. The case $\alpha=0$ corresponds
to the classic Ornstein-Uhlenbeck process.

\subsubsection*{Power -$\mathcal{BSS}$ processes}

In \cite{BenLundPakk16A} the authors considered the $\mathcal{BSS}$
associated to the function
\[
\varphi_{\alpha}\left(x\right)=\frac{x^{\alpha}}{(1+x)^{\alpha+\beta}},\,\,\,\alpha\in(-1/2,0)\cup(0,1/2),\beta>1/2,
\]
and they referred to it as a {\it power}-$\mathcal{BSS}$ process.
A remarkable property of this family of $\mathcal{BSS}$ processes
is that, on the one hand, the parameter $\alpha$ controls the path
regularities of $Y$, while on the other, $\beta$ characterizes
the persistence of $Y$. More precisely, $Y$ has $p$-Hölder-continuous
paths almost surely for any $p<\alpha+1/2$ and its autocovariance
function satisfies that as $h\uparrow\infty$
\[
\gamma(h)\sim\begin{cases}
K_{\sigma,\beta,\alpha}h^{1-2\beta} & \text{if }\beta\in(1/2,1);\\
K_{\sigma,\beta,\alpha}h^{-\beta} & \text{if }\beta>1.
\end{cases}
\]
for some constant $K_{\sigma,\beta,\alpha}>0$. For a proof of these
properties see \cite{BenLundPakk16}. Therefore, $Y$ has long memory
in the case when $\beta\in(1/2,1)$, while when $\beta>1$, the process
has short memory. The case $\beta=1$ is a limit case.

\subsection{Fractional Brownian motion}

A continuous centered Gaussian process $\left(B_{t}^{H}\right)_{t\in\mathbb{R}}$
is called fractional Brownian motion (fBm for short) with Hurst index
$H\in\left(0,1\right),$ if its covariance structure is given by 
\[
\mathbb{E}\left(B_{t}^{H}B_{s}^{H}\right)=\frac{c_{H}}{2}\left(\left\vert t\right\vert ^{2H}+\left\vert s\right\vert ^{2H}-\left\vert t-s\right\vert ^{2H}\right)\text{, \ \ for any }s,t\in\mathbb{R}\text{,}
\]
with $c_{H}=\int_{0}^{\infty}\left[\left(1+x\right)^{H-1/2}-x^{H-1/2}\right]dx+1/2H$.
It is well known that $B^{H}$ has (up to modification) Hölder continuous
paths of order strictly smaller than $H$ and it is a semimartingale
if and only if $H=1/2,$ i.e. when $B^{H}$ is a Brownian motion.
Furthermore, $B^{H}$ admits the following spectral representation
\begin{equation}
B_{t}^{H}=\int_{\mathbb{R}}K^{H}\left(t,s\right)dB_{s}\text{, \ \ }t\in\mathbb{R}\text{,}\label{eqnchap11.3}
\end{equation}
where 
\[
K^{H}\left(t,s\right):=\left(t-s\right)_{+}^{H-1/2}-\left(-s\right)_{+}^{H-1/2}\text{, \ \ }s,t\in\mathbb{R}\text{,}
\]
with $\left(x\right)_{+}:=\max\left\{ 0,x\right\} $. It follows from
this that $B^{H}$ is not adapted to the filtration generated by the
increments of $B$ for $t<0$ and it is a Volterra-type representation
for $t>0$. We claim that for $t\geq0$
\begin{equation}
B_{t}^{H}=A_{t}+\widetilde{B}_{t}^{H},\label{eqnchap1.4.1}
\end{equation}
where $\widetilde{B}_{t}^{H}=\int_{0}^{t}\left(t-s\right)^{H-1/2}dB_{s}$
and $A_{t}:=\int_{-\infty}^{0}K^{H}\left(t,s\right)dB_{s}$ being
a process with absolutely continuous paths. Indeed, since for any
$t\geq0$
\[
K^{H}\left(t,s\right)=(H-1/2)\int_{0}^{t}(u-s)^{H-3/2}du,\,\,\,s<0,
\]
and 
\begin{align*}
\int_{0}^{t}(\int_{-\infty}^{0}(u-s)^{2H-3}ds)^{1/2}du & =\frac{1}{H(2-2H)^{1/2}}t^{H},
\end{align*}
the Stochastic Fubini Theorem can be applied (see Theorem \ref{StFubiniThm}
in the Appendix) to deduce that for any $t\geq0$, almost surely
\[
A_{t}:=(H-1/2)\int_{0}^{t}\int_{-\infty}^{0}(u-s)^{H-3/2}dB_{s}du,
\]
meaning this that, up to a modification, the paths of $A_{t}$ are
absolutely continuous for $t\geq0$.

More generally, any process admitting the integral representation
\begin{equation}
B_{t}^{\sigma,H}=\int_{\mathbb{R}}K^{H}\left(t,s\right)\sigma_{s}dB_{s},\text{ \ \ }t\in\mathbb{R}\text{,}\label{eqnchap11.4.0}
\end{equation}
for an adapted càglàd process $\sigma$, will be referred to as a
\textit{volatility modulated fractional Brownian motion} (vmfBm for
short). Observe that by a localization argument on $\sigma$, we can
construct a modification of $B^{\sigma,H}$ which has Hölder continuous
paths of index strictly smaller that $H$. A similar decomposition
to (\ref{eqnchap1.4.1}) can be obtained in the case when $\sigma$
is bounded in $\mathcal{L}^{2}\left(\Omega,\mathcal{F},\mathbb{P}\right)$
.

\section{Wiener integrals for volatility modulated Volterra processes}

In this section, by reasoning as in \citet{Mocioalca04} as well as
in \citet{BNBEnthPedVeraat14}, we define Wiener-type stochastic integrals
with respect to volatility modulated Volterra processes on the real
line. Such an integral will play a key role in our main result.

\subsection{Definition and basic properties}

Let $K:\mathbb{R}^{2}\rightarrow\mathbb{R}$ denotes a measurable
function such that $K\left(t,s\right)=0$ for $s>t$, $\int_{-\infty}^{t}K^{2}\left(t,s\right)ds<\infty$
for any $t\in\mathbb{R}$, and the mapping $t\mapsto K(t,s)$ is continuously
differentiable on $(s,\infty)$ for almost all $s$. A stochastic
process $\left(X_{t}\right)_{t\in\mathbb{R}}$ is called \textit{volatility
modulated Volterra process} ($\mathcal{VMVP}$ for short) on the real
line if it admits the representation
\begin{equation}
X_{t}=\int_{-\infty}^{t}K\left(t,s\right)\sigma_{s}dB_{s},\text{ \ \ \ }t\in\mathbb{R}\text{,}\label{eqnchap11.9}
\end{equation}
for some adapted càglàd process that is bounded in $\mathcal{L}^{2}(\Omega,\mathcal{F},\mathbb{P})$.
Observe that $X$ is well defined because almost surely $\int_{-\infty}^{t}K\left(t,s\right)^{2}\sigma_{s}^{2}ds<\infty$.
Indeed, the $\mathcal{L}^{2}$-boundedness of $\sigma$ guarantees
the existence of a constant $M>0$ such that
\[
\mathbb{E}[\int_{-\infty}^{t}K\left(t,s\right)^{2}\sigma_{s}^{2}ds]\leq M\int_{-\infty}^{t}K^{2}\left(t,s\right)ds<\infty,\,\,\,\forall\,t\in\mathbb{R},
\]
as claimed. The function $K$ will be referred as the kernel of $X$.
For simplicity, for the rest of this paper we will assume that $K(0,s)=0$
for almost all $s$. We will also assume that $K(t,\cdot) \xrightarrow[t \to u]{}K(u,\cdot)$ in $\mathcal{L}^2(ds)$. This in particular
implies that $X$ is continuous in probability (see \cite{BasseGravPed13}) and therefore it admits a measurable modification.

Let us give some remarks:

\begin{remark} We observe the following:
\begin{enumerate}
\item In general, $X$ is not adapted to the natural filtration of $B$
but rather to the filtration generated by the increments of $B$.
\item The fractional Brownian motion can be written as the sum of an $\mathcal{VMVP}$
and an increment semimartingale. Indeed, from (\ref{eqnchap11.3}),
we have that
\[
B_{t}^{H}=\int_{-\infty}^{t}K^{H}\left(t,s\right)dB_{s}-\int_{t\land0}^{0}\left(-s\right)^{H-1/2}dB_{s},\,\,\,t\in\mathbb{R},
\]
where, as usual, the symbol $t\land0$ represents the minimum between
$t$ and $0$.
\item In general, $\mathcal{VMVP}$ are neither semimartingales nor increment
semimartingales. 
\end{enumerate}
\end{remark}

From the previous remark it follows that standard Itô's stochastic
integration cannot be applied to $\mathcal{VMVP}$. Nevertheless,
proceeding as in \citet{BNBEnthPedVeraat14}, c.f. \citealt{Mocioalca04},
we introduce a Wiener-type stochastic integral with respect to (w.r.t.
for short) $\mathcal{VMVP}$ as follows: Suppose for a moment that
$K$ is smooth enough to make $X$ a semimartingale (see general conditions
in Proposition 3 of \citet{BNBEnthVeraat13}), then by putting $dB_{t}^{\sigma}=\sigma_{t}dB_{t}$,
we have 
\begin{equation}
dX_{t}=K\left(t,t-\right)dB_{t}^{\sigma}+\int_{-\infty}^{t}\frac{\partial}{\partial t}K\left(t,s\right)dB_{s}^{\sigma}dt,\text{ \ \ \ }t\in\mathbb{R}\text{.}\label{eqnchap11.10.0}
\end{equation}
In general, $K\left(t,t-\right)$ might not be well defined, e.g.
$K$ as in (\ref{eqnchap11.3}), but even when it is, one can have
that $\frac{\partial}{\partial t}K\left(t,\cdot\right)$ is not integrable
with respect to $B^{\sigma}$. However, we can formally infer from
(\ref{eqnchap11.10.0}) that
\begin{eqnarray}
\int_{-\infty}^{t}f\left(s\right)dX_{s} & = & \int_{-\infty}^{t}\int_{-\infty}^{s}f\left(s\right)\frac{\partial}{\partial s}K\left(s,r\right)dB_{r}^{\sigma}ds+\int_{-\infty}^{t}f\left(s\right)K\left(s,s-\right)dB_{s}^{\sigma}\label{eqnchap11.10}\\
 & = & \int_{-\infty}^{t}\int_{-\infty}^{s}\left[f\left(s\right)-f\left(r\right)\right]\frac{\partial}{\partial s}K\left(s,r\right)dB_{r}^{\sigma}ds+\int_{-\infty}^{t}f\left(s\right)K\left(s,s-\right)dB_{s}^{\sigma}\nonumber \\
 &  & +\int_{-\infty}^{t}\int_{-\infty}^{s}f\left(r\right)\frac{\partial}{\partial s}K\left(s,r\right)dB_{r}^{\sigma}ds,\nonumber 
\end{eqnarray}
but, by the stochastic Fubini theorem 
\[
\int_{-\infty}^{t}\int_{-\infty}^{s}\left[f\left(s\right)-f\left(r\right)\right]\frac{\partial}{\partial s}K\left(s,r\right)dB_{r}^{\sigma}ds=\int_{-\infty}^{t}\int_{s}^{t}\left[f\left(u\right)-f\left(s\right)\right]\frac{\partial}{\partial u}K\left(u,r\right)dudB_{s}^{\sigma},
\]
and
\begin{eqnarray*}
\int_{-\infty}^{t}\int_{-\infty}^{s}f\left(r\right)\frac{\partial}{\partial s}K\left(s,r\right)dB_{r}^{\sigma}ds & = & \int_{-\infty}^{t}f\left(r\right)\left[K\left(t,r\right)-K\left(r,r-\right)\right]dB_{r}^{\sigma}.
\end{eqnarray*}
Thus, in the semimartingale case, we have that
\begin{eqnarray}
\int_{-\infty}^{t}f\left(s\right)dX_{s} & = & \int_{-\infty}^{t}\mathcal{K}_{K}f\left(t,s\right)\sigma_{s}dB_{s},\label{eqnchap11.10.1}
\end{eqnarray}
where 
\begin{equation}
\mathcal{K}_{K}f\left(t,s\right):=f\left(s\right)K\left(t,s\right)+\int_{s}^{t}\left[f\left(u\right)-f\left(s\right)\right]\frac{\partial}{\partial u}K\left(u,s\right)du\text{, \ \ \ }s<t.\label{eqnchap11.11}
\end{equation}
Note that the left-hand side of (\ref{eqnchap11.10.1}) is well defined
(remember that $\sigma$ is bounded in $\mathcal{L}^{2}(\Omega,\mathcal{F},\mathbb{P})$)
if $\mathcal{K}_{K}f(t,\cdot$) is square integrable, which a priori
does not requires that $K(t,t-)<\infty$. Thus, (\ref{eqnchap11.10.1})
gives a way to define stochastic integrals w.r.t. $\mathcal{VMVP}$
even in the non-semimartingale case. More precisely: 

\begin{definition} \label{Defintstoch2}Let $X$ be an $\mathcal{VMVP}$
as in (\ref{eqnchap11.9}) and put
\[
\mathcal{H}_{t}^{K}:=\left\{ f\text{ measurable}:\mathcal{K}_{K}f\in\mathcal{L}^{2}\left(\left(-\infty,t\right]\right)\right\} ,\text{ \ \ \ }t\in\mathbb{R}.
\]
For each $f\in\mathcal{H}_{t}^{K},$ we define the stochastic integral
of $f$ w.r.t. $X$ as
\begin{equation}
\int_{-\infty}^{t}f\left(s\right)dX_{s}:=\int_{-\infty}^{t}\mathcal{K}_{K}\left(f\right)\left(t,s\right)dB_{s}^{\sigma}\text{, \ \ \ }t\in\mathbb{R}.\label{eqnchap114}
\end{equation}
\end{definition}

\begin{remark} Observe that $\int_{-\infty}^{t}f\left(s\right)dX_{s}$
is a particular case of the one introduced in \citet{AlosMazerNualart01},
\citet{Mocioalca04}, and \citet{BNBEnthPedVeraat14} for non-random
integrands with support on $\mathbb{R}$ rather than in an interval
of the form $\left[0,T\right]$. Furthermore, $\mathcal{H}_{t}^{K}$
is a Hilbert space with inner product 
\[
\left\langle f,g\right\rangle _{\mathcal{H}_{t}^{K}}=\left\langle \mathcal{K}_{K}f,\mathcal{K}_{K}g\right\rangle _{\mathcal{L}^{2}\left(\left(-\infty,t\right]\right)}.
\]
\end{remark}

\begin{remark} We might be tempted to extend Definition \ref{Defintstoch2}
to any predictable process $Y$ by putting 
\[
\int_{-\infty}^{t}Y_{s}dX_{s}=\int_{-\infty}^{t}\mathcal{K}_{K}\left(Y\right)\left(t,s\right)dB_{s}^{\sigma}\text{.}
\]
However, under this definition, the process $\left(\mathcal{K}_{K}\left(Y\right)\left(t,s\right)\right)_{s\leq t}$
is not adapted because the random variable \linebreak $\int_{s}^{t}\left[Y_{u}-Y_{s}\right]\frac{\partial}{\partial u}K\left(u,r\right)du$
is only $\mathcal{F}_{t}$-measurable. Hence, we cannot define such
an integral in the Itô's sense, but it is feasible in the Skorohod
sense as it was done in \citet{BNBEnthPedVeraat14}. 

\end{remark}

The following example shows that there is a natural connection between
the stochastic integral defined in (\ref{eqnchap114}) and $\mathcal{BSS}$
processes of the form of (\ref{eqnchap11.1}). It is also the starting
point and motivation for our main result in the next section.

\begin{example}[$\mathcal{BSS}$ and $\mathcal{VMVP}$ ] \label{BSSgammandVMVP}For
$\alpha\in(-1/2,0)\cup(0,1/2)$ let $K$ be the Mandelbrot-Van Ness
kernel, i.e.
\begin{equation}
K\left(t,s\right)=\{\left(t-s\right)_{+}^{\alpha}-\left(-s\right)_{+}^{\alpha}\}\mathbf{1}_{\left\{ t>s\right\} }\text{,}\label{kernelmandvaness-1}
\end{equation}
and consider 
\[
f(s)=e^{\lambda s},\,\,\,s\in\mathbb{R},\lambda>0.
\]
Let us see that $\int_{-\infty}^{t}e^{\lambda s}dX_{s}$ can be defined
as in Definition \ref{Defintstoch2} for every $t$. First note that
\begin{equation}
\mathcal{K}_{K}\left(f\right)(t,s)=e^{\lambda t}K\left(t,s\right)-\lambda\int_{s}^{t}e^{\lambda u}K\left(u,s\right)du,\,\,s<t.\label{Kopexp}
\end{equation}
Moreover, due to the Jensen's inequality and Tonelli's Theorem we
have that
\[
\int_{-\infty}^{t}\left(\int_{s}^{t}e^{\lambda u}K\left(u,s\right)du\right)^{2}ds\leq\frac{e^{\lambda t}}{\lambda}\int_{-\infty}^{t}\int_{s}^{t}e^{\lambda u}K\left(u,s\right)^{2}duds=\frac{c_{\alpha}e^{\lambda t}}{\lambda}\int_{-\infty}^{t}e^{\lambda u}u^{2\alpha+1}du<\infty,
\]
with $c_{\alpha}$ a constant depending only on $\alpha$. This, together
with (\ref{Kopexp}) shows that $\int_{-\infty}^{t}e^{\lambda s}dX_{s}$
is well defined and by linearity so $\int_{-\infty}^{t}e^{-\lambda(t-s)}dX_{s}$ is.
Further, for $t\geq0$ almost surely
\begin{align}
\int_{-\infty}^{t}e^{-\lambda(t-s)}dX_{s} & -Y_{t}=e^{\lambda t}\{\int_{-\infty}^{t}\int_{s}^{t}[e^{\lambda u}-e^{\lambda s}]\frac{\partial}{\partial u}K\left(u,s\right)du\sigma_{s}dB_{s}-\int_{-\infty}^{0}e^{\lambda s}\left(-s\right)^{\alpha}\sigma_{s}dB_{s}\},\label{decompgammkernel}
\end{align}
where $Y$ is a $\mathcal{BSS}$ with a gamma kernel. Observe that
this decomposition is well defined. Our main result in the next section
states that the process in the right-hand side of (\ref{decompgammkernel})
has absolutely continuous paths almost surely. Hence, $\mathcal{BSS}$ processes with gamma kernel and $\mathcal{VMVP}$
differs only by an absolutely continuous process.

\end{example}

The previous example motivates a further understanding of the integral
introduced in Definition \ref{Defintstoch2}. Therefore, for the rest
of this section we study some properties of the integral appearing
in (\ref{eqnchap114}). Let us start with the connection between such
an integral and the Lebesgue-Stieltjes integral:

\begin{proposition} \label{proplebstiel}Let $f:\mathbb{R}\rightarrow\mathbb{R}$
be a continuous function of local bounded variation satisfying that
for almost all $s<t$
\begin{equation}
\lim_{u\downarrow s}\left[f\left(u\right)-f\left(s\right)\right]K\left(u,s\right)=0.\label{eqnchap1143}
\end{equation}
Then $f\in\mathcal{H}_{t}^{K}$ if and only if the mapping $s\mapsto\int_{s}^{t}K(r,s)df(r)$
is square integrable. If in addition we have that
\begin{equation}
\int_{-\infty}^{t}\left(\int_{-\infty}^{s}K(s,r)^2dr\right)^{1/2}d\left\vert f\right\vert \left(s\right)<\infty,\label{eqnchap1142}
\end{equation}
where $d\left\vert f\right\vert $ is the total variation of $f$,
then almost surely
\begin{equation}
\int_{-\infty}^{t}f\left(s\right)dX_{s}=f\left(t\right)X_{t}-\int_{-\infty}^{t}X_{s}df\left(s\right),\label{LebStieltHtH0}
\end{equation}
i.e. $\int_{-\infty}^{t}f\left(s\right)dX_{s}$ coincides with the
Lebesgue-Stieltjes integral of $f$ with respect to $X$.

\end{proposition}

\begin{proof} Integration by part shows that for any continuous function
of local bounded variation satisfying equation (\ref{eqnchap1143})
it holds that 
\[
\int_{s}^{t}\left[f\left(u\right)-f\left(s\right)\right]\frac{\partial}{\partial u}K\left(u,s\right)du=\left[f\left(t\right)-f\left(s\right)\right]K\left(t,s\right)-\int_{r}^{t}K\left(u,s\right)df\left(u\right),
\]
for $s<t$. Thus 
\begin{equation}
\mathcal{K}_{K}f\left(t,s\right)=f\left(t\right)K\left(t,s\right)-\int_{s}^{t}K(r,s)df(r)\text{, \ \ \ }s<t.\label{operatorH_0}
\end{equation}
The first conclusion of this proposition follows by noting that $\int_{-\infty}^{t}K^{2}\left(t,s\right)ds<\infty$
for any $t\in\mathbb{R}$ and using the fact that $\mathcal{H}_{t}^{K}$
is a linear space. Moreover, due to (\ref{operatorH_0}), in the case
when $f\in\mathcal{H}_{t}^{K}$ satisfies (\ref{eqnchap1143}) we
have that
\[
\int_{-\infty}^{t}f(s)dX_{s}=f(t)X_{t}-\int_{-\infty}^{t}\int_{s}^{t}K\left(r,s\right)df(r)dB_{s}^{\sigma}.
\]
Suppose for a moment that the stochastic Fubini theorem can be applied
to $\int_{-\infty}^{t}\int_{r}^{t}K\left(s,r\right)\sigma df(r)dB_{r}^{\sigma},$
then almost surely
\begin{eqnarray*}
\int_{-\infty}^{t}f(s)dX_{s} & = & f(t)X_{t}-\int_{-\infty}^{t}\int_{s}^{t}K\left(r,s\right)df(r)dB_{s}^{\sigma}.\\
 & = & f\left(t\right)X_{t}-\int_{-\infty}^{t}X_{r}df(r),
\end{eqnarray*}
which shows (\ref{LebStieltHtH0}). Hence, we only need to show that
under our assumptions we can swap the order of integration in $\int_{-\infty}^{t}\int_{s}^{t}K\left(r,s\right)df\left(r\right)dB_{s}^{\sigma}.$
Since $f$ is locally of bounded variation, it can be written as $f=g_{1}-g_{2},$
where $g_{1}$ and $g_{2}$ are non-decreasing functions. By \ref{eqnchap1142}, we deduce that for $i=1,2$, $\int_{-\infty}^{t}\left(\int_{-\infty}^{s}K(s,r)^2dr\right)^{1/2}dg_{i}\left(s\right)<\infty$, which, thanks to Lemma \ref{lemmafubini} in the Appendix, implies that almost surely
for $i=1,2$
\[
\int_{-\infty}^{t}X_{s}dg_{i}\left(s\right)=\int_{-\infty}^{t}\int_{r}^{t}K\left(s,r\right)dg_{i}dB_{r}^{\sigma},
\]
which completes the proof.

\end{proof}

\begin{remark}We would like to remark that since the mapping $u\mapsto\left\Vert K(t,\cdot)-K(u,\cdot)\right\Vert _{\mathcal{L}^{2}\left(ds\right)}$
is continuous, then $C_{b}^{1},$ the space of continuously differentiable
functions with compact support, belongs to $\mathcal{H}_{t}^{K}$
just as in Proposition 14 in \citet{Mocioalca04}. 

\end{remark}

\begin{remark}

Observe that a similar condition to (\ref{eqnchap1143}) was imposed
in \citealt{BNBEnthPedVeraat14} to determine the class of solutions
associated to the Langevin equation driven by an $\mathcal{VMVP}$
indexed on the positive real line. 

\end{remark}

\begin{remark} \label{RemarkLebSt}Equation (\ref{eqnchap1143})
is trivial when $f$ is continuous and $K(t+,t)$ exists for every
$t$. Furthermore, when $\int_{-\infty}^{t}d\left\vert f\right\vert(s)<\infty$,
the condition 
\begin{equation}
\int_{-\infty}^{t}\int_{-\infty}^{s}K(s,r)^2dr d\left\vert f\right\vert (s)<\infty
\end{equation}
together with (\ref{eqnchap1143}), imply that all
the conclusions of Proposition \ref{proplebstiel} hold for such $f$. Indeed, by the
Jensen's inequality and Tonelli's Theorem
\[
\left[\int_{-\infty}^{t}\left(\int_{-\infty}^{s}K(s,r)^2dr\right)^{1/2}d\left\vert f\right\vert (s)\right]^{2}\leq c_{t}\int_{-\infty}^{t}\int_{-\infty}^{s}K(s,r)^2dr d\left\vert f\right\vert (s)<\infty,
\]
as well as
\[
\int_{\infty}^{t}\left|\int_{s}^{t}K(r,s)df(r)\right|^{2}ds\leq c_{t}\int_{-\infty}^{t}\int_{s}^{t}K(r,s)^{2}d\left\vert f\right\vert \left(r\right)ds=c_{t}\int_{-\infty}^{t}\int_{-\infty}^{s}K(s,r)^2dr d\left\vert f\right\vert (s),
\]
where $c_{t}=\int_{-\infty}^{t}d\left\vert f\right\vert \left(s\right)$.

\end{remark}
The next example concentrates in the case when the kernel $K$ corresponds
to the one of the fBm.

\begin{example} \label{madelbrovanness}Let $K$ be as in Example
	\ref{BSSgammandVMVP}. Observe that when $\alpha>0$, (\ref{eqnchap1143})
	is satisfied for any continuous function. On the other hand, when
	$\alpha<0$ we have that every function of bounded variation satisfies
	(\ref{eqnchap1143}) for almost all $s\in\mathbb{R}$. Indeed, if
	$f$ is of bounded variation then its derivative exists almost everywhere,
	so pick $s\in\mathbb{R}\backslash\{0\}$ for which the derivative of $f$ exists.
	Then
	\[
	\lim_{u\downarrow s}[f\left(u\right)-f\left(s\right)]K(u,s)=\lim_{u\downarrow s}(\frac{f(u)-f(s)}{u-s})K(u,s)\left(u-s\right)=0\text{.}
	\]
	where we have used that $\frac{f(u)-f(s)}{u-s}\rightarrow f'(s)$
	and $K(u,s)\left(u-s\right)\rightarrow0$ as $u\downarrow s$. Moreover, in view that
	\[ \int_{-\infty}^{s}K(s,r)^2dr=c_\alpha s^{2\alpha +1} ,\]
	for some $c_\alpha>0$, then condition (\ref{eqnchap1142}) becomes
	\[ \int_{-\infty}^{t}s^{\alpha+1/2}d\left\vert f\right\vert (s)<\infty.\]
	In particular, if for $s\leq t$ we let 
	\begin{align*}
	f_{t,\lambda}(s):= & e^{-\lambda(t-s)},\,\,\,\lambda>0;\\
	f_{t,\beta,\alpha}(s):= & (1+t-s)^{-(\alpha+\beta)},\,\,\,\beta>1/2,
	\end{align*}
	then $f_{t,\lambda}$ and $f_{t,\beta,\alpha}$ satisfy (\ref{eqnchap1143}) and (\ref{eqnchap1142}).
	In Example \ref{BSSgammandVMVP}, we have showed that $f_{t,\lambda}\in\mathcal{H}_{t}^{K}$, meaning this that almost surely
	\begin{equation}
	\int_{-\infty}^{t}e^{-\lambda(t-s)}dX_{s}  =X_{t}-\lambda e^{-\lambda t}\int_{-\infty}^{t}e^{\lambda s}X_{s}ds.\label{langevineq}
	\end{equation}
	Later we will see that also $f_{t,\beta,\alpha}$ belongs to $\mathcal{H}_{t}^{K}$, which according to Proposition \ref{proplebstiel} lead us to the almost surely relation
	\begin{equation*}
	\int_{-\infty}^{t}(1+t-s)^{-(\alpha+\beta)}dX_{s} =X_{t}-(\alpha+\beta)\int_{-\infty}^{t}(1+t-s)^{-\alpha-\beta-1}X_{s}ds.
	\end{equation*}
	
\end{example}

Next, for the sake of completeness, we state the following result
related to the continuity of $\int_{-\infty}^{t}\cdot dX_{s}$ whose
proof is straightforward and thus omitted.

\begin{proposition} The mapping $I_{t}:\mathcal{H}_{t}^{K}\rightarrow\mathcal{L}^{2}\left(\Omega,\mathcal{F},\mathbb{P}\right),$
defined by 
\[
I_{t}\left(f\right):=\int_{-\infty}^{t}f\left(s\right)dX_{s},
\]
is a continuous linear isometry. \end{proposition}

\begin{remark} Observe that from \citet{BasseGravPed12}, in general,
we have that
\[
\{I_{t}\left(f\right):f\in\mathcal{H}_{t}^{K}\}\subsetneq\overline{span}\{X_{u}-X_{v}:v\leq u\leq t\},
\]
where $\overline{span}$ represents the closed linear span on $\mathcal{L}^{2}\left(\Omega,\mathcal{F},\mathbb{P}\right)$.
See also \citet{Jolis10} and \citet{PipiTaqqu00}.

\end{remark}

\subsection{The Langevin equation and OU processes driven by an $\mathcal{VMVP}$}

Let $X$ be an $\mathcal{VMVP}$ satisfying the assumptions of the
previous subsection. Recall that a process $\left(Z_{t}\right)_{t\geq0}$
solves the Langevin equation with parameter $\lambda>0$ with respect
to $X$ if and only if for $t\geq0$ almost surely 
\begin{equation}
Z_{t}=\xi+X_{t}-\lambda\int_{0}^{t}Z_{s}ds\text{,}\label{eqnchap11.12}
\end{equation}
for some random variable $\xi$ which is $\mathcal{F}_{0}$-measurable.
Thus, by looking carefully into (\ref{langevineq}) in Example \ref{madelbrovanness}
one can actually see that the process $(\int_{-\infty}^{t}e^{-\lambda(t-s)}dX_{s})_{t\geq0}$
solves the Langevin equation associated to such an $\mathcal{VMVP}$. In this case, the initial condition is 
$\xi=-\lambda\int_{-\infty}^{0}e^{\lambda u}X_{u}du$. In this part,
we generalize such a result to the context of Proposition \ref{proplebstiel}. 

For $\lambda>0$, set 
\[
f_{t,\lambda}(s):=e^{-\lambda\left(t-s\right)},\text{ \ \ \  }s\leq t\text{.}
\]

\begin{proposition} \label{Proplangevineq}Suppose that $f_{t,\lambda}$
satisfies the assumptions of Proposition \ref{proplebstiel} for every
$t\geq0$. Then the process
\begin{equation}
Z_{t}:=\int_{-\infty}^{t}e^{-\lambda(t-s)}dX_{s}\text{, \ \ }t\geq0\text{.}\label{eqnchap11.13}
\end{equation}
is the unique solution (up to a modification) of the Langevin equation
with parameter $\lambda>0$ and initial condition $\xi=-\lambda\int_{-\infty}^{0}e^{\lambda u}X_{u}du$. 

\end{proposition}

\begin{proof} Let us first show uniqueness. Let $Z^{1}$ and $Z^{2}$
be two solutions of the Langevin equation. Then the process
\[
H_{t}:=Z_{t}^{1}-Z_{t}^{2}=-\lambda\int_{0}^{t}H_{s}ds,\text{ \ \ }t\geq0,
\]
satisfies the ordinary differential equation
\[
\frac{d}{dt}H_{t}=-\lambda H_{t},\text{ \ \ }t\geq0\text{,}
\]
with initial condition $H_{0}=0$. This implies necessarily that $H_{t}=0$
almost surely for every $t\geq0$. The uniqueness of the solution
follows from this.

Now, if for every $t\geq0$, $f_{\lambda}\left(t,\cdot\right)$ satisfies
the assumptions of Proposition \ref{proplebstiel} we have that $Z$
is well defined and almost surely 
\begin{equation}
Z_{t}=X_{t}-\lambda e^{-\lambda t}\int_{-\infty}^{t}e^{\lambda s}X_{s}ds.\label{decomp}
\end{equation}
This, together wit Lemma \ref{lemmafubini}, show that $Z$ has a modification that is locally
Lebesgue integrable which also will be denoted by $Z$. Furthermore,
integration by parts gives that 
\[
-\lambda\int_{0}^{t}e^{-\lambda u}\int_{-\infty}^{u}e^{\lambda s}X_{s}dsdu=e^{-\lambda t}\int_{-\infty}^{t}e^{\lambda s}X_{s}ds-\int_{-\infty}^{0}e^{\lambda s}X_{s}ds-\int_{0}^{t}X_{u}du.
\]
Consequently, by (\ref{decomp}), for every $t\geq0$, almost surely
\[
-\lambda\int_{0}^{t}Z_{u}du=-\lambda[e^{-\lambda t}\int_{-\infty}^{t}e^{\lambda s}X_{s}ds-\int_{-\infty}^{0}e^{\lambda s}X_{s}ds]=Z_{t}-\xi-X_{t},
\]
as required. \end{proof}

\begin{remark} Note that the previous result is in agreement with \citet{BasseBN11},
Theorem 2.1, and Proposition 8 in \citet{BNBEnthPedVeraat14} in the
context of $\mathcal{VMVP}$. However, in general, the process $Z$
is not stationary.

\end{remark}

We conclude this section with an example that links the process appearing
in (\ref{eqnchap11.13}) and the fractional Ornstein-Uhlenbeck process.

\begin{example}[Fractional OU processes]In this example the processes
$B$ and $\sigma$ are fixed, with $\sigma$ bounded in $\mathcal{L}^{2}(\Omega,\mathcal{F},\mathbb{P})$. Let $X$ be as in Example \ref{BSSgammandVMVP}. As noticed in the
beginning of this section, any volatility modulated fBm can be written
as 
\begin{equation}
B_{t}^{H,\sigma}=X_{t}-\int_{t\land0}^{0}\left(-s\right)^{H-1/2}\sigma_{s}dB_{s},\,\,\,t\in\mathbb{R},\label{vmfbmdynamics}
\end{equation}
where $H=\alpha+1/2.$ Since the process $A_{t}=\int_{t\land0}^{0}\left(-s\right)^{H-1/2}\sigma_{s}dB_{s}$
is an increment semimartingale, the integral $\int_{-\infty}^{t}f(s)dA_{s}$
is well defined in the Itô's sense if, but in general not only if,
$\int_{-\infty}^{t}f^{2}\left(s\right)\left(-s\right)_{+}^{2H-1}ds<\infty$.
We may thus define the integral of $f$ w.r.t. to $B^{H,\sigma}$ as 
\begin{equation}
\int_{-\infty}^{t}f(s)dB_{s}^{H,\sigma}=\int_{-\infty}^{t}f(s)dX_{s}+\int_{-\infty}^{t}f(s)dA_{s}.\label{stintfbm}
\end{equation}
We have seen in Example \ref{madelbrovanness} that $f_{t,\lambda}$
satisfies the assumptions of Proposition \ref{proplebstiel}. In view
of this and the fact that $\int_{-\infty}^{t}e^{-2\lambda(t-s)}\left(-s\right)_{+}^{2H-1}ds<\infty$,
the process $Z_{t}^{H}=\int_{-\infty}^{t}e^{-\lambda(t-s)}dB_{s}^{H,\sigma}$
is well defined in the sense of (\ref{stintfbm}). When $\sigma$
is constant, $Z$ coincides with the usual fractional Ornstein-Uhlenbeck
process which was introduced in \citet{Cheridito2003}. The most remarkable
property of $Z^{H}$ is that the process
\[
U_{t}:=Z_{t}^{H}-\int_{-\infty}^{t}e^{-\lambda(t-s)}dX_{s}=e^{-\lambda t}\int_{-\infty}^{t}e^{\lambda s}\left(-s\right)_{+}^{H-1/2}\sigma_{s}dB_{s},
\]
is of unbounded variation for $t<0,$ while for $t\geq0$ it has absolutely
continuous paths almost surely.

\end{example}

\section{A pathwise decomposition of $\mathcal{BSS}$ with respect to vmfBm
and its application}

Through several examples, in the previous section we showed that there
is a strong connection between $\mathcal{BSS}$ processes and a subclass
of $\mathcal{VMVP}$. In this section, such a connection is explicitly
derived when we consider $\mathcal{BSS}$ of the form of
\begin{equation}
Y_{t}:=\int_{-\infty}^{t}\varphi_{\alpha}\left(t-s\right)\sigma_{s}dB_{s},\text{ \ \ }t\in\mathbb{R}\text{,}\label{eqnchap11.1-1}
\end{equation}
where 
\begin{equation}
\varphi_{\alpha}\left(x\right)=L\left(x\right)x^{\alpha},\text{ \ \ }x>0,\label{eqnchap11.2-1}
\end{equation}
$\alpha\in(-1/2,0)\cup(0,1/2)$ and $\sigma$ an adapted stationary
càglàd process which is bounded in $\mathcal{L}^{2}(\Omega,\mathcal{F},\mathbb{P})$.
More precisely, under certain regularities on $L$, we show that $Y$
and the $\mathcal{VMVP}$ given by
\begin{equation}
X_{t}=\int_{-\infty}^{t}[(t-s)_{+}^{\alpha}-(-s)_{+}^{\alpha}]\sigma_{s}dB_{s}\label{FractionalVMVP}
\end{equation}
 differ only by an absolutely continuous process. 

\subsection{Assumptions and statement of the result}

Let us start by introducing our working assumption which is a refinement
of that considered in \cite{BenLundPakk16A}. We recall to the reader
that the notation $f(x)=O(g(x))$ as $x\rightarrow c$ means that
$\limsup_{x\rightarrow c}\left|f(x)/g(x)\right|<\infty.$

\begin{assumption} \label{assumption1}The function $L$ in (\ref{eqnchap11.2-1})
satisfies the following:
\begin{enumerate}
\item $L$ is twice continuously differentiable on $[0,\infty)$ such that
$L(0)\neq0$ and $\int_{0}^{\infty}\left|L(s)s^{\alpha}\right|^{2}ds<\infty.$ 
\item There is $\alpha+3/2<\zeta_{0}$ such that as $s\uparrow\infty$,
$L'(s)=O(s^{-\zeta_{0}})$ and $L''(s)=O(s^{-(\zeta_{0}+1)})$.
\end{enumerate}
\end{assumption}

Under this assumption we have that:

\begin{theorem} \label{mainthm} Let Assumption \ref{assumption1}
holds and consider $X$ as in (\ref{FractionalVMVP}). Then $L(t-\cdot)$
is integrable w.r.t. $X$ for every $t\in\mathbb{R}$. Moreover, if
$Y$ belongs to the class of $\mathcal{BSS}$ appearing in (\ref{eqnchap11.1-1})
and
\[
Y_{t}^{X}:=\int_{-\infty}^{t}L\left(t-s\right)dX_{s},\text{ \ 		\ }t\in\mathbb{R},
\]
then the process
\[
V_{t}:=Y_{t}^{X}-Y_{t},\,\,\,t\geq0
\]
has absolutely continuous paths almost surely.

\end{theorem}

\begin{proof} Along the proof the constant $c_\alpha>0$ is such that $\int_{-\infty}^{t}K(t,s)^2ds=c_\alpha t^{2\alpha+1}$. Firstly we note that under Assumption \ref{assumption1}, $L,L'$ and $L''$ are totally bounded and satisfy (\ref{eqnchap1143}). Thus, thanks to Proposition \ref{proplebstiel} and Lemma \ref{lemmaHtk} below, $L\left(t-\cdot\right)\in\mathcal{H}_{t}^{K}$
for any $t\in\mathbb{R}$. Furthermore, Assumption \ref{assumption1} also 
guarantees the existence of $\alpha+3/2<\zeta$
such that for any $x\geq0$ 
\begin{align}
\left|L(x)\right| & \leq M_{0}\left|x\right|^{-(\zeta-1)};\nonumber\\
\left|L'(x)\right| & \leq M_{1}\left|x\right|^{-\zeta};\label{estimatesmainthm}\\
\left|L''(x)\right| & \leq M_{2}\left|x\right|^{-(\zeta+1)}\nonumber.
\end{align}
for some constants $M_{0},M_{1},M_{2}>0$. In view of this and since $L$ is twice continuously
differentiable on $[0,\infty)$ we have that for any $-\infty<s_{0}<t\wedge0$
\begin{equation}
\int_{-\infty}^{t}s^{\alpha+1/2}\left|L'(t-s)\right|ds \leq\int_{s_{0}}^{t}s^{\alpha+1/2}\left|L'(t-s)\right|ds+M_{1}\int_{-\infty}^{s_{0}}s^{\alpha+1/2}\left|t-s\right|^{-\zeta}ds<\infty\label{eq1proofthm}.
\end{equation}
Therefore, from Example \ref{madelbrovanness} and Proposition \ref{proplebstiel}, almost surely
\[
L(0)X_{t}=Y_{t}^{X}-\int_{-\infty}^{t}L'(t-s)X_{s}ds.
\]
On the other hand, reasoning as in (\ref{eq1proofthm}) we obtain that $\int_{-\infty}^t\arrowvert L(t-s)s^\alpha\lvert^2ds<\infty$. Moreover, since $L$ is totally bounded we deduce that the mapping $s\mapsto L(t-s)K(t,s)$ is square integrable.
Hence, we get that almost surely
\[
L(0)X_{t}=Y_{t}+\int_{-\infty}^{t}\left[L\left(0\right)-L\left(t-s\right)\right]K\left(t,s\right)\sigma_{s}dB_{s}-\int_{-\infty}^{t}L(t-s)(-s)_{+}^{\alpha}\sigma_{s}dB_{s}.
\]
All in all imply that for $t\geq0$ almost surely
\[
V_{t}=U_{t}^{1}+U_{t}^{2}+U_{t}^{3},
\]
where
\begin{align*}
U_{t}^{1}:= & \int_{-\infty}^{t}K_{L}\left(t,s\right)\sigma_{s}dB_{s},\\
U_{t}^{2}:= & \int_{-\infty}^{t}[L(t-s)-L(0)]dX_s,\\
U_{t}^{3}:= & -\int_{-\infty}^{0}L(t-s)(-s)^{\alpha}\sigma_{s}dB_{s},
\end{align*}
with 
\[
K_{L}\left(t,s\right):=\left[L\left(0\right)-L\left(t-s\right)\right]K\left(t,s\right)\text{, \ \ }t>s.
\]
Hence, in order to finish the proof, it is enough to show that for $i=1,2,3$, $U^{i}$
has almost surely absolutely continuous paths. Suppose for a
moment that the following processes are well defined 
\begin{align*}
u_{t}^{1}:= & \int_{-\infty}^{t}\frac{\partial}{\partial t}K_{L}\left(t,s\right)\sigma_{s}dB_{s},\\
u_{t}^{2}:= &L'(0)X_t+ \int_{-\infty}^{t}L''(t-s)X_{s}ds,\\
u_{t}^{3}:= & -\int_{-\infty}^{0}L'(t-s)(-s)^{\alpha}\sigma_{s}dB_{s},
\end{align*}
and that the Fubini Theorem (stochastic and non-stochastic) can be
applied. Then, we necessarily have that 
\[
\int_{0}^{t}u_{s}^{i}ds=U_{t}^{i}-U_{0}^{i},\,\,\,t\geq0,
\]
which would conclude the proof. Therefore, in what follows, we check
that for $i=1,2,3$, $u^{i}$ is well defined and that the Fubini
theorem can be applied. 

Fix $t\geq0>s_0>-\infty$.
The estimates in (\ref{estimatesmainthm}) and the fact
that $L'$ and $L''$ are totally bounded by, let's say $M',M''>0$, respectively, give us that
\begin{align}
\int_{-\infty}^{t}\left|L'(t-s)K\left(t,s\right)\right|^{2}ds & \leq M'c_\alpha t^{2\alpha+1}<\infty;\nonumber \\
\int_{-\infty}^{t}\left|\left[L\left(0\right)-L\left(t-s\right)\right]\frac{\partial}{\partial t}K\left(t,s\right)\right|^{2}ds & =\int_{0}^{\infty}\left|\left[L\left(0\right)-L\left(r\right)\right]r^{\alpha-1}\right|^{2}ds<\infty;\label{estimates}\\
\int_{-\infty}^{0}\left|L'(t-s)(-s)^{\alpha}\right|^{2}ds & \leq M'\int_{s_0}^{0}s^{2\alpha}ds+M_1\int_{-\infty}^{s_0}s^{2(\alpha-\zeta)}ds<\infty,\nonumber 
\end{align}
which shows the well definiteness of $u^{1}$ and $u^{3}$, while
\begin{equation}
\int_{-\infty}^{t}\arrowvert L''(t-s)s^{\alpha+1/2}\lvert ds\leq M''\int_{s_0}^{t}s^{\alpha+1/2}ds+M_1\int_{-\infty}^{s_0}s^{(\alpha-\zeta)-1/2}ds<\infty,\label{secondderestimates}
\end{equation}
together with Proposition \ref{proplebstiel} and Lemma \ref{lemmaHtk} below, show that $u^{2}$ is well defined and admits the representation
\begin{equation}
u^2_t=\int_{-\infty}^tL'(t-s)dX_s. \label{lbiprimrepre}
\end{equation}
Therefore, it only rests to check that we can apply the stochastic
Fubini Theorem. By using (\ref{estimates}) and applying the triangle inequality
one obtain that
\begin{equation*}
\int_{0}^{t}(\int_{-\infty}^{u}\frac{\partial}{\partial u}K_{L}\left(u,s\right)^{2}ds)^{1/2}du \leq(M'c_\alpha)^{1/2}\int_{0}^{t}s^{\alpha+1/2}du+(\int_{0}^{\infty}\left|\left[L\left(0\right)-L\left(r\right)\right]r^{\alpha-1}\right|^{2}ds)^{1/2}t<\infty.
\end{equation*}
which according to Theorem \ref{StFubiniThm} justifies the interchange
of the Itô's integral with the Lebesgue integral for $u^{1}$. Finally, thanks to (\ref{estimates}), (\ref{secondderestimates}) and the Dominated Convergence Theorem, we deduce that the mappings $u\mapsto\int_{-\infty}^{0}\left|L'(u-s)(-s)^{\alpha}\right|^{2}ds$ and $u\mapsto\int_{-\infty}^{u}\arrowvert L''(u-s)s^{\alpha+1/2}\lvert ds$ are continuous on $[0,\infty)$. This in particular implies that 
\[ \int_{0}^{t}\int_{-\infty}^{0}\left|L'(u-s)(-s)^{\alpha}\right|^{2}dsdu<\infty,\text{ and }\int_{0}^{t}\int_{-\infty}^{u}\arrowvert L''(u-s)s^{\alpha+1/2}\lvert dsdu<\infty, \]
which, according to Remark \ref{remarkfubini}, justifies the application of the stochastic Fubini theorem for $u^{2}$ and $u^{3}.$\end{proof}

\begin{lemma} \label{lemmaHtk}Fix $t\in\mathbb{R}$. For every continuous
function $f$ on $(-\infty,t]$, let
\[
\ell_{t}(s):=\int_{s}^{t}\{\left(r-s\right)_{+}^{\alpha}-\left(-s\right)_{+}^{\alpha}\}f(r)dr,\,\,\,s<t.
\]
 If $f(u)=O(\left|u\right|^{-\zeta})$ as $u\downarrow-\infty$, for
some $\alpha+3/2<\zeta$, then
\[
\int_{-\infty}^{t}\ell_{t}(s)^{2}ds<\infty.
\]
\end{lemma}

\begin{proof}

Let $c_\alpha>0$ be as in the proof of Theorem  \label{mainthm}. Firstly, we claim that $\ell_{t}$ is locally square integrable on
$(-\infty,t]$ whenever $f$ is . Indeed, let $-\infty<s_{0}<t$.
Then by the Jensen's inequality and Tonelli's Theorem, it holds that
\begin{align*}
(\int_{s_{0}}^{t}\left|f(u)\right|du)^{-1}\int_{s_{0}}^{t}\left|\ell_{t}(s)\right|^{2}ds & \leq\int_{s_{0}}^{t}\int_{s_{0}}^{r}\{\left(r-s\right)_{+}^{\alpha}-\left(-s\right)_{+}^{\alpha}\}^{2}ds\left|f(r)\right|dr\leq c_{\alpha}\int_{s_{0}}^{t}r^{2\alpha+1}\left|f(r)\right|dr<\infty,
\end{align*}
which proves our claim. Consequently, we only need to show that under our hypothesis on $f$, $\ell_{t}$ is square integrable at infinity.
Let $s<(t\wedge0)$, then 
\begin{equation}
\ell_{t}(s)=\begin{cases}
\left|s\right|^{\alpha+1}\int_{0}^{1}[(1-r)^{\alpha}-1]f(sr)dr-\int_{t}^{0}\{\left(r-s\right)^{\alpha}-\left(-s\right)^{\alpha}\}f(r)dr & \text{if }t\leq0;\\
\left|s\right|^{\alpha+1}\int_{0}^{1}[(1-r)^{\alpha}-1]f(sr)dr+\int_{0}^{t}\{\left(r-s\right)^{\alpha}-\left(-s\right)^{\alpha}\}f(r)dr & \text{if }t>0.
\end{cases}\label{ldecomposition}
\end{equation}
We show first that in general for any continuous $f$, it holds that
\begin{equation}
\ell_{t}(s)-\left|s\right|^{\alpha+1}\int_{0}^{1}[(1-r)^{\alpha}-1]f(sr)dr=O(\left|s\right|^{\alpha-1}),\,\,\,s\downarrow-\infty.\label{negliblelt}
\end{equation}
By the mean value theorem, we have that for $t>0$, 
\[
\int_{0}^{t}\left|\{\left(r-s\right)^{\alpha}-\left(-s\right)^{\alpha}\}f(r)\right|dr\leq\left|\alpha\right|\left|s\right|^{\alpha-1}\int_{0}^{t}\left|rf(r)\right|dr,
\]
which implies trivially (\ref{negliblelt}). In an analogous way, we get that for $s<t\leq0$
\[
\int_{t}^{0}\left|\{\left(r-s\right)^{\alpha}-\left(-s\right)^{\alpha}\}f(r)\right|dr\leq\left|\alpha\right|\left|s\right|^{\alpha-1}\int_{t}^{0}\left(1-\frac{r}{s}\right)^{\alpha-1}\left|rf(r)\right|dr\leq\left|\alpha\right|\left(1-\frac{t}{s}\right)^{\alpha-1}\left|s\right|^{\alpha-1}\int_{t}^{0}\left|rf(r)\right|dr,
\]
from which (\ref{negliblelt}) can be deduced. All above implies
that if $f$ is continuous, then $\ell_{t}$ is square integrable if
and only if the mapping $s\mapsto\left|s\right|^{\alpha+1}\int_{0}^{1}[(1-r)^{\alpha}-1]f(sr)dr$
is square integrable at infinity. Let us show that under our hypothesis
this holds. Since $f$ is continuous on $(-\infty,t]$ and $f(u)=O(u^{-\zeta})$
for some $\alpha+3/2<\zeta$ at $-\infty$, then for any $\alpha+3/2<\tilde{\zeta}\leq\zeta$,
we can choose a constant $M>0$ such that
for any $x\leq0$ 
\[
\left|f(x)\right|\leq M\left|x\right|^{-\tilde{\zeta}},
\]
where we use the convention that $0^{-1}=+\infty.$ In particular,
if we choose $\tilde{\zeta}$ such that $\alpha+3/2<\tilde{\zeta}<\min\{2,\zeta\}$
then for any $s<0$ 
\begin{align*}
\left|s\right|^{\alpha+1}\left|\int_{0}^{1}[(1-r)^{\alpha}-1]f(sr)dr\right| & \leq M\left|s\right|^{\alpha+1-\tilde{\zeta}}\int_{0}^{1}\left|(1-r)^{\alpha}-1\right|r^{-\tilde{\zeta}}dr.
\end{align*}
Since $\tilde{\zeta}<2$, we have that $\int_{0}^{1}\left|(1-r)^{\alpha}-1\right|r^{-\tilde{\zeta}}dr<\infty$.
The desired square integrability follows from the previous estimates
and the fact that $\alpha+1-\tilde{\zeta}<-1/2.$\end{proof}

\begin{corollary} \label{fBmdecompcoro}Under the assumptions of
Theorem \ref{mainthm}, we have that for every $t\in\mathbb{R}$,
almost surely 
\[
Y_{t}=L(0)B_{t}^{H,\sigma}+U_{t}+A_{t},
\]
where $Y$ is a $\mathcal{BSS}$ process of the form of (\ref{eqnchap11.1-1}),
$B^{H,\sigma}$ is a volatility modulated fractional Brownian motion
with $H=\alpha+1/2$, $U$ has absolutely continuous paths, and $A$
is an increment semimartingale for $t<0$ and it vanishes for $t\geq0$.
Moreover, for $t\geq0$, the following hold
\begin{enumerate}
\item $(Y_{t})_{t\geq0}$ is a semimartingale if and only if $\alpha=0$
and has a version which is Hölder continuous with index $\rho<\alpha+1/2;$
\item The $p$-variation of $Y$ over $\left[0,T\right]$
\[
V_{p}^{Y}\left(\left[0,T\right]\right)^{p}:=\sup_{\pi:\pi\text{ is a partition of }\left[0,T\right]}\left\{ \sum_{t_{k}\in\pi}\left\vert Y_{t_{k}}-Y_{t_{k-1}}\right\vert ^{p}\right\} ,
\]
is finite almost surely for every $p>\frac{1}{\alpha+1/2}$. In particular,
the quadratic variation of $Y$ vanishes for $\alpha>0$ and is not
finite for $\alpha<0$. 
\end{enumerate}
\end{corollary}

\begin{proof} From the proof of Theorem \ref{mainthm}, we have that
almost surely 
\[
L(0)X_{t}=Y_{t}-U_{t}.
\]
where $U$ is a process with absolutely continuous paths. Moreover,
from (\ref{vmfbmdynamics}), we have that 
\[
L(0)X_{t}=L(0)[B_{t}^{H,\sigma}+\int_{t\land0}^{0}\left(-s\right)^{H-1/2}\sigma_{s}dB_{s}],\,\,\,t\in\mathbb{R},
\]
The conclusion of this corollary follows by letting
\begin{align*}
A_{t}:=L(0) & \int_{t\land0}^{0}\left(-s\right)^{H-1/2}\sigma_{s}dB_{s},\,\,\,t\in\mathbb{R}.
\end{align*}
\end{proof}

\begin{example}Let
\begin{align*}
L_{\lambda}(x)= & e^{-\lambda x}\\
L_{\beta,\alpha}= & (1+x)^{-(\alpha+\beta)},
\end{align*}
for $\lambda>0$, $\alpha\in(-1/2,0)\cup(0,1/2)$ and $\beta>1/2$.
Since
\begin{align*}
\lim_{x\downarrow\infty}L_{\lambda}'(x)x^{p} & =\lim_{x\downarrow\infty}L_{\lambda}''(x)x^{p+1}=0;\\
\lim_{x\downarrow\infty}L'_{\beta,\alpha}(x)x^{\zeta_{0}}= & \lim_{x\downarrow\infty}L''_{\beta,\alpha}(x)x^{(\zeta_{0}+1)}=0,
\end{align*}
for any $p\geq0$ and $\alpha+3/2<\zeta_{0}<\alpha+\beta+1$ we have
that $L_{\lambda}\text{ and }L_{\beta,\alpha}$ satisfy Assumption
\ref{assumption1}. This means that for $t\geq0$, $\mathcal{BSS}$
processes with a gamma kernel and power-$\mathcal{BSS}$ differ from
a vmfBm only by an absolutely continuous process. It is interesting
to note that the former type of $\mathcal{BSS}$ has short memory,
meaning this that $L_{\lambda}$ kills all the persistence coming
from the vmfBm. An analogous reasoning can be done to the power-$\mathcal{BSS}$
when $\beta>1$, while for $1/2<\beta<1$ the memory introduced by
the vmfBm is not negligible anymore.

\end{example}

\subsection{Itô's formulae for $\mathcal{BSS}$ processes}

In the previous subsection we have seen that any $\mathcal{BSS}$
of the form of (\ref{eqnchap11.1}) are not in general semimartingales
and they are in fact closely related to fractional Brownian motions.
In this subsection, by using some stochastic integrals (with random
integrands) for the fractional Brownian motion, we derive some Itô's
formulae. Observe that, as in the case of the fBm, stochastic integrals
(with random integrands) can be defined in several ways. Let us note
that the goal of this subsection is only to show an application of
Theorem \ref{mainthm}. We do not pretend to develop stochastic integrals
with respect to $Y$ for random integrands, but rather use the well-known
results concerning to the fBm. For a survey in stochastic calculus
for fBm we refer to \citep{Coutin07}.

\paragraph{Itô's formula based on Young integrals}

Let $T>0.$ Consider $Y$ as in (\ref{eqnchap11.1}). Within the framework
of Corollary \ref{fBmdecompcoro}, the $p$-variation of $Y$ is finite
almost surely for every $p>\frac{1}{\alpha+1/2}$. Therefore, if $\left(Z_{t}\right)_{t\geq0}$
is a continuous process with finite $q$-variation, with $q<\frac{1}{\alpha-1/2}$
and $\alpha>0$, we have that the following Riemann sums
\[
\sum\limits _{i=0}^{n-1}Z_{t_{i}^{n}}\left(Y_{t_{i+1}^{n}}-Y_{t_{i}^{n}}\right),
\]
converges almost surely. Here, $0=t_{0}^{n}<t_{1}^{n}<\cdots<t_{n}^{n}=T$
and are such that $\sup\left\vert t_{i+1}^{n}-t_{i}^{n}\right\vert \rightarrow0$
as $n\rightarrow\infty$. The limit does not depend on the partition
$\left(t_{i}^{n}\right).$ See \citet{Young36} for more details.
Hence, we may define $\int_{0}^{T}Z_{s}dY_{s}$, the stochastic integral
of $Z$ with respect to $Y,$ as such limit. Using this result we
obtain the following Itô's formula based on \citet{Young36}.

\begin{theorem} Let $Y$ be as in (\ref{eqnchap11.1-1}) with $0<\alpha<1/2.$
For every $f:\mathbb{R\rightarrow}\mathbb{R}$ continuously differentiable
function with Hölder continuous derivative of order $\beta>\frac{1}{\alpha-1/2}-1$
the following Itô's formula holds
\[
f\left(Y_{T}\right)=f\left(Y_{0}\right)+\int_{0}^{T}f^{\prime}\left(Y_{s}\right)dY_{s}\text{,}
\]
where $\int_{0}^{T}f^{\prime}\left(Y_{s}\right)dY_{s}$ is understood
pathwise. \end{theorem}

\paragraph{Itô's formula based on Malliavin calculus}

In this part, for simplicity, we will assume that $L\left(0\right)=1$
and $\sigma\equiv1$. Due to Theorem \ref{mainthm}, we can define
stochastic integrals (with random integrands) with respect to $Y$
as follows: under the assumptions of Theorem \ref{mainthm}, we have
that for any $t\geq0$
\[
Y_{t}=U_{t}+\widetilde{B}_{t}^{\alpha},
\]
with $U$ an absolutely continuous process and $\widetilde{B}_{t}^{\alpha}:=\int_{0}^{t}\left(t-s\right)^{\alpha}dB_{s}.$
Thus, if $\mathcal{F=\sigma}\left(\widetilde{B}^{\alpha}\right)$
and $\left(Z_{t}\right)_{0\leq t\leq T}$ is a continuous $\mathcal{F}_{T}$-measurable
process, we define the Skorohod integral of $Z$ with respect to $Y$
as
\begin{equation}
\int_{0}^{T}Z_{s}\delta Y_{s}:=\int_{0}^{T}Z_{s}dU_{s}+\int_{0}^{T}Z_{s}\delta\widetilde{B}_{s}^{\alpha},\label{eqnchap11.19}
\end{equation}
where, in the nomenclature of \citet{Nualart06}, $\int_{0}^{T}Z_{s}\delta\widetilde{B}_{s}^{\alpha}$
represents the divergence operator evaluated at $Z.$ Let $f:\mathbb{R\rightarrow}\mathbb{R}$
be a twice continuously differentiable function. Then, by Taylor's
Theorem, for $0=t_{0}^{n}<t_{1}^{n}<\cdots<t_{n}^{n}=T$ with $\sup\left\vert t_{i+1}^{n}-t_{i}^{n}\right\vert \rightarrow0$
as $n\rightarrow\infty,$ we get
\begin{eqnarray*}
f\left(Y_{T}\right)-f\left(Y_{0}\right) & = & \sum\limits _{i=0}^{n-1}f^{\prime}\left(Y_{t_{i}^{n}}\right)\left(Y_{t_{i+1}^{n}}-Y_{t_{i}^{n}}\right)\\
 &  & +\frac{1}{2}\sum\limits _{i=0}^{n-1}f^{\prime\prime}\left(Y_{t_{i}^{n}}\right)\left(Y_{t_{i+1}^{n}}-Y_{t_{i}^{n}}\right)^{2}+\sum\limits _{i=0}^{n-1}R\left(Y_{t_{i+1}^{n}},Y_{t_{i}^{n}}\right)
\end{eqnarray*}
where, from Theorem 1 and Proposition 5 in \citet{AlosMazerNualart01},
as $n\rightarrow\infty$ the series $\sum\limits _{i=0}^{n-1}f^{\prime}\left(Y_{t_{i}^{n}}\right)\left(Y_{t_{i+1}^{n}}-Y_{t_{i}^{n}}\right)$
converges to 
\begin{equation}
\int_{0}^{T}f^{\prime}\left(Y_{s}\right)\delta Y_{s}+\alpha\int_{0}^{T}\int_{s}^{T}D_{s}f^{\prime}\left(Y_{u}\right)\left(u-s\right)^{\alpha-1}duds,\label{eqnchap11.19.0}
\end{equation}
in $\mathcal{L}^{2}\left(\Omega,\mathcal{F},\mathbb{P}\right)$, provided
that for some constants $c>0$ and for every $\zeta>0$ such that
$\zeta<\left(4\int_{0}^{\infty}\varphi_{\alpha}^{2}\left(s\right)ds\right)^{-1}$,
it holds that 
\begin{equation}
\max\left\{ \left\vert f\left(x\right)\right\vert ,\left\vert f^{\prime}\left(x\right)\right\vert ,\left\vert f^{\prime\prime}\left(x\right)\right\vert \right\} \leq ce^{\zeta\left\vert x\right\vert ^{2}},\label{eqnchap11.20}
\end{equation}
plus the technical condition 
\begin{equation}
\lim_{n\rightarrow\infty}\int_{0}^{T}\sup_{u,v\in\left(s,s+1/n\right]\cap\left[0,T\right]}\mathbb{E}\left[\left\vert D_{s}f^{\prime}\left(Y_{u}\right)-D_{s}f^{\prime}\left(Y_{v}\right)\right\vert ^{2}\right]ds=0,\label{eqnchap11.21}
\end{equation}
with $D_{s}$ denoting the Malliavin derivative induced by $\widetilde{B}^{\alpha}$.
See for instance \citet{Nualart06}. On the other hand, from (\ref{eqnchap11.20})
and Corollary \ref{fBmdecompcoro}, for $\alpha>0$ we get that $\sum\limits _{i=0}^{n-1}f^{\prime\prime}\left(Y_{t_{i}^{n}}\right)\left(Y_{t_{i+1}^{n}}-Y_{t_{i}^{n}}\right)^{2}\rightarrow0$
in $\mathcal{L}^{1}\left(\Omega,\mathcal{F},\mathbb{P}\right)$. Furthermore,
by standard arguments, it can be shown that the remaining term $\sum\limits _{i=0}^{n-1}R\left(Y_{t_{i+1}^{n}},Y_{t_{i}^{n}}\right)$
tends to $0$ in $\mathcal{L}^{1}\left(\Omega,\mathcal{F},\mathbb{P}\right)$.
Hence, as a corollary to Theorem \ref{mainthm} and \citet{AlosMazerNualart01}
we have that

\begin{theorem} Let $Y$ be as in (\ref{eqnchap11.1-1}) with $0<\alpha<1/2$
and $\sigma\equiv1$. For every $f$ of class $C^{2}$ satisfying
(\ref{eqnchap11.20}) and (\ref{eqnchap11.21}), we have that the
following Itô's formula applies
\begin{equation}
f\left(Y_{T}\right)=f\left(Y_{0}\right)+\int_{0}^{T}f^{\prime}\left(Y_{s}\right)\delta Y_{s}+\alpha\int_{0}^{T}\int_{s}^{T}D_{s}f^{\prime}\left(Y_{u}\right)\left(u-s\right)^{\alpha-1}duds,\label{itomalliavin}
\end{equation}
where $\int_{0}^{T}f^{\prime}\left(Y_{s}\right)\delta Y_{s}$ is understood
as in (\ref{eqnchap11.19}). \end{theorem}

\begin{remark} Note that the formula in (\ref{itomalliavin}) differs
from the one in \citet{AlosMazerNualart01} by the term 
\[
\alpha\int_{0}^{T}\int_{s}^{T}D_{s}f^{\prime}\left(Y_{u}\right)\left(u-s\right)^{\alpha-1}duds.
\]
 However, if we define 
\begin{equation}
\int_{0}^{T}Z_{s}dY_{s}:=\int_{0}^{T}Z_{s}\delta Y_{s}+\int_{0}^{T}D_{s}\left\{ \mathcal{K}_{K}\left[Y\left(T,s\right)\right]\right\} ds,\label{eqnchap11.22}
\end{equation}
where $\mathcal{K}_{K}$ is the operator in (\ref{eqnchap11.11})
with $K\left(t,s\right)=\left(t-s\right)^{\alpha},$ then (\ref{itomalliavin})
can be written as
\[
f\left(Y_{T}\right)=f\left(Y_{0}\right)+\int_{0}^{T}f^{\prime}\left(Y_{s}\right)dY_{s}.
\]
Thus, the integral in (\ref{eqnchap11.22}) coincides with that in
\citet{BNBEnthPedVeraat14}. \end{remark}

\begin{remark} Observe that (\ref{eqnchap11.19.0}) still holds for
$0>\alpha>-1/4$. Nevertheless, at this point, it is not clear what
is the limit behavior of $\sum\limits _{i=0}^{n-1}R\left(Y_{t_{i+1}^{n}},Y_{t_{i}^{n}}\right)$
in this case.\end{remark}

\section{Conclusions}

In this paper, by using Wiener-type stochastic integral for volatility
modulated Volterra processes on the real line, we have decomposed
a subclass of $\mathcal{BSS}$ as a sum of a fractional Brownian motion
and an absolutely continuous process. We exploit this decomposition
in order to obtain the index of Hölder continuity of the $\mathcal{BSS}$
of interest. Furthermore, we derived Itô's formula in the case when
$1/2>\alpha>0$. 

\subsubsection*{Acknowledgement}

The author wishes to thank Jan Pedersen for fruitful ideas and discussions
which made the present work possible. Further, he gratefully acknowledges
Ole E. Barndorff-Nielsen, Benedykt Szozda and Lorenzo Boldrini for
helpful comments on an earlier draft of this article.

\section{Appendix}

\subsection*{Stochastic Fubini Theorem}

Our main result is based mainly on the stochastic Fubini theorem for
semimartingales. In the following, we present a review of some conditions
for which this theorem holds. See \citep{Veraar12} for a detailed
discussion and \citep{BasseBN11} for generalizations.

Let $T$ be an interval on $\mathbb{R}$ and $\left(\mathcal{X},\mathcal{B},\mu\right)$
be a $\sigma$-finite measure space. Consider a real-valued random
field $\psi:\Omega\times T\times\mathcal{X\rightarrow}\mathbb{R}$.
Assume that $\psi$ is jointly measurable and adapted on $T$. Let
$\left(S_{t}\right)_{t\in T}$ be a continuous semimartingale on $\left(\Omega,\mathcal{F},\left(\mathcal{F}_{t}\right)_{t\in\mathbb{R}},\mathbb{P}\right)$.
The stochastic Fubini theorem gives necessary conditions for which
almost surely
\begin{equation}
\int_{\mathcal{X}}\int_{T}\psi\left(t,x\right)dS_{t}\mu\left(dx\right)=\int_{T}\int_{\mathcal{X}}\psi\left(t,x\right)\mu\left(dx\right)dS_{t},\label{eqnchap11.5}
\end{equation}
in which implicitly is also concluded that the Lebesgue and Itô integrals
exist. Let $\left(M,A\right)$ be the canonical decomposition of $S$,
i.e. $M$ is a continuous local martingale, $A$ a continuous process
of bounded variation and $S=M+A$. We have that (\ref{eqnchap11.5})
holds under the assumptions of the following theorem whose proof can
be found in \citep{Veraar12}:

\begin{theorem}[Stochastic Fubini Theorem] \label{StFubiniThm}Let
$\psi:\Omega\times T\times\mathcal{X\rightarrow}\mathbb{R}$ be progressively
measurable on $T$ and assume that almost surely 
\begin{eqnarray}
 &  & \int_{\mathcal{X}}\left(\int_{T}\psi^{2}\left(s,x\right)d\left[M\right]_{s}\right)^{1/2}\mu\left(dx\right)<\infty;\label{eqnchap11.6.1}\\
 &  & \int_{\mathcal{X}}\int_{T}\psi^{2}\left(s,x\right)\left\vert dA\right\vert _{s}\mu\left(dx\right)<\infty,\label{eqnchap11.6.2}
\end{eqnarray}
where $\left[M\right]$ denotes the quadratic variation of $M$ and
$\left\vert dA\right\vert $ the total variation of the Lebesgue-Stieltjes
measure associated to $A$. Then (\ref{eqnchap11.5}) holds almost
surely . \end{theorem}

\begin{remark} \label{remarkfubini}Observe that when $\mu$ is finite, (\ref{eqnchap11.6.1})
in the previous theorem can be replaced by 
\begin{equation}
\int_{\mathcal{X}}\int_{T}\psi^{2}\left(s,x\right)d\left[M\right]_{s}\mu\left(dx\right)<\infty\text{.}\label{eqnchap11.7}
\end{equation}
For the particular case of the Brownian motion it is equivalent to
\begin{equation}
\int_{\mathcal{X}}\left(\int_{T}\psi^{2}\left(s,x\right)ds\right)^{1/2}\mu\left(dx\right)<\infty\text{.}\label{eqnchap11.8}
\end{equation}
Finally, when $\psi$ is square integrable, (\ref{eqnchap11.6.1})
can be replaced by 
\begin{equation}
\int_{\mathcal{X}}\left[\mathbb{E}\left(\int_{T}\psi^{2}\left(s,x\right)d\left[M\right]_{s}\right)\right]^{1/2}\mu\left(dx\right)<\infty,\label{eqnchap11.7.1}
\end{equation}
or 
\begin{equation}
\int_{\mathcal{X}}\mathbb{E}\left(\int_{T}\psi^{2}\left(s,x\right)d\left[M\right]_{s}\right)\mu\left(dx\right)<\infty,\label{eqnchap11.8.1}
\end{equation}
when $\mu$ is finite. \end{remark}
The following result is a version of Theorem \ref{StFubiniThm} into the context of $\mathcal{VMVP}$.
\begin{lemma}\label{lemmafubini}Let $(\mathcal{X},\mathcal{G})$ be a measurable space and $\mu$ a $\sigma$-finite measure on $\mathcal{G}\otimes\mathcal{B}(\mathbb{R})$. Consider $(X_t)_{t\in \mathbb{R}}$ to be an $\mathcal{VMVP}$ satisfying the assumptions of Section 2. If $f:\mathcal{X}\times\mathbb{R}\rightarrow \mathbb{R}$ is a measurable function fulfilling the condition
	\[ \int_{\mathcal{X}}\int_{\mathbb{R}}\arrowvert f(x,s)\lvert\left(\int_{-\infty}^{t}K(s,u)^2du\right)^{1/2}\mu(dxds)<\infty, \]
	then almost surely $\int_{\mathcal{X}}\int_{\mathbb{R}}\arrowvert f(x,s)X_s\lvert\mu(dxds)<\infty$ and
	\[ \int_{\mathcal{X}}\int_{\mathbb{R}} f(x,s)X_s\mu(dxds)=\int_{\mathbb{R}}\int_{\mathcal{X}}K(s,r)f(x,s)\mu(dxds)\sigma_rdB_r. \]
\end{lemma}

\subsection*{Proof of Proposition \ref{stationarity}}

In this part we give a proof of Proposition \ref{stationarity}. We
would like to emphasize that the proof, and consequently the Proposition,
is valid for any Lévy process.

\begin{proof}{[}Proof of Proposition \protect\ref{stationarity}{]}
As it was mentioned in the beginning of this paper, for the $\mathcal{BSS}$
process to be well defined, we must assume that $\int_{0}^{\infty}\phi\left(s\right)^{2}\sigma_{s+t}^{2}ds<\infty$
almost surely for any $t\in\mathbb{R}$. Firstly, let us show that
for all $h\in\mathbb{R}$, almost surely
\begin{equation}
\int_{-\infty}^{t+h}\phi\left(t+h-s\right)\sigma_{s}dB_{s}=\int_{-\infty}^{t}\phi\left(t-s\right)\sigma_{s+h}dB_{s}^{h}.\label{eqnchap11}
\end{equation}
Indeed, since $\sigma$ is predictable and $\phi$ is measurable,
without loss of generality we may and do assume that $\varphi_{\alpha}$
is non-negative and 
\[
\sigma_{t}=\sum\limits _{i=1}^{N}a_{i}\mathbf{1}_{\left(u_{i},t_{i}\right]}\left(t\right)X_{i}\text{, \ \ }t\in\mathbb{R}\text{,}
\]
where $a_{i}\in\mathbb{R}$, $-\infty<u_{i}<t_{i}<\infty$ and $X_{i}\in\mathcal{F}_{u_{i}}$.
Hence, we can choose a sequence of simple functions $\left(\phi^{n}\right)_{n\in\mathbb{N}}$
with $\phi^{n}\uparrow\phi$ and 
\[
\phi^{n}\left(s\right)=\sum\limits _{i=1}^{m_{n}}b_{i}^{n}\mathbf{1}_{\left\{ p_{i}\leq s<q_{i}\right\} },\text{ \ \ }s\in\mathbb{R}\text{,}
\]
where $b_{i}^{n}\in\mathbb{R}^{+}$ and $0\leq p_{i}<q_{i}<\infty$.
We see that $\int_{-\infty}^{t}\phi^{n}\left(t-s\right)\sigma_{s}dB_{s}\overset{\mathbb{P}}{\rightarrow}Y_{t}$
for the reason that $\phi^{n}\left(t-\cdot\right)\sigma$ is a sequence
of simple predictable processes converging to $\varphi_{\alpha}\left(t-\cdot\right)\sigma$.
Observe that 
\[
\int_{-\infty}^{t+h}\phi^{n}\left(t+h-s\right)\sigma_{s}dB_{s}=\sum\limits _{i=1}^{N}\sum\limits _{j=1}^{m_{n}}a_{i}b_{j}^{n}X_{i}\int_{-\infty}^{t+h}\mathbf{1}_{\left\{ u_{i}\vee t+h-q_{i}<s\leq t_{i}\wedge t+h-p_{i}\right\} }dB_{s},
\]
and 
\[
\int_{-\infty}^{t}\phi^{n}\left(t-s\right)\sigma_{s+h}dB_{s}^{h}=\sum\limits _{i=1}^{N}\sum\limits _{j=1}^{m_{n}}a_{i}b_{j}^{n}X_{i}\int_{-\infty}^{t}\mathbf{1}_{\left\{ u_{i}-h\vee t-q_{i}<s\leq t_{i}-h\wedge t-p_{i}\right\} }dB_{s}^{h}.
\]
Since
\begin{eqnarray*}
\int_{-\infty}^{t+h}\mathbf{1}_{\left\{ u_{i}\vee t+h-q_{i}<s\leq t_{i}\wedge t+h-p_{i}\right\} }dB_{s} & = & B_{t_{i}\wedge t+h-p_{i}}-B_{u_{i}\vee t+h-q_{i}}\\
 & = & \int_{-\infty}^{t}\mathbf{1}_{\left\{ u_{i}-h\vee t-q_{i}<s\leq t_{i}-h\wedge t-p_{i}\right\} }dB_{s}^{h},
\end{eqnarray*}
we get 
\[
\int_{-\infty}^{t+h}\phi^{n}\left(t+h-s\right)\sigma_{s}dB_{s}=\int_{-\infty}^{t}\phi^{n}\left(t-s\right)\sigma_{s+h}dB_{s}^{h}\text{.}
\]
Taking limits in the last equation we obtain (\ref{eqnchap11}).

Thanks to (\ref{eqnchap11}), in order to finish the proof, we only
need to prove that 
\[
\left(\int_{-\infty}^{t}\phi\left(t-s\right)\sigma_{s+h}dB_{s}^{h}\right)_{t\in\mathbb{R}}\overset{d}{=}\left(\int_{-\infty}^{t}\phi\left(t-s\right)\sigma_{s}dB_{s}\right)_{t\in\mathbb{R}}\text{.}
\]
Since $\sigma$ is càglàd, we may and do assume that it is bounded,
therefore for any $h\in\mathbb{R}$ 
\begin{eqnarray*}
\phi\left(t-s\right)\sigma_{s+h} & = & \lim_{n\rightarrow\infty}\sum\limits _{k\in\mathbb{Z}}\phi^{n}\left(t-s\right)\sigma_{\frac{k}{2^{n}}+h}\mathbf{1}_{\left\{ \frac{k}{2^{n}}<s\leq\frac{k+1}{2^{n}}\right\} }\\
 & = & :\lim_{n\rightarrow\infty}\psi_{h}^{n}\left(t-s\right)\text{,}
\end{eqnarray*}
where the limit is pointwise. Note that $\left\vert \psi_{h}^{n}\right\vert \leq\phi\left(t-\cdot\right)$,
this jointly with (\ref{eqnchap11}) give $\int_{-\infty}^{t}\psi_{h}^{n}\left(t-s\right)dB_{s}^{h}\overset{\mathbb{P}}{\rightarrow}Y_{t+h}$.
It only remains to show that 
\[
\left(\int_{-\infty}^{t}\psi_{h}^{n}\left(t-s\right)dB_{s}^{h}\right)_{t\in\mathbb{R}}\overset{d}{=}\left(\int_{-\infty}^{t}\psi_{0}^{n}\left(t-s\right)dB_{s}\right)_{t\in\mathbb{R}}\text{.}
\]
Indeed, if it were true, we would have that $\int_{-\infty}^{.}\psi_{h}^{n}\left(t-s\right)dB_{s}^{h}\Longrightarrow Y$,
where $\Longrightarrow$ denotes convergence in finite dimensional
distributions, which gives the desired result. Now, observe that (\ref{eqnchap13.1})
implies that for any $f:\mathbb{R}^{2\times M}\mathbb{\rightarrow R}$
continuous function and $\left(r_{1},\ldots,r_{M}\right)\in\mathbb{R}^{M}$
\begin{equation}
f\left(\left\{ \left(\sigma_{r_{i}+h},B_{r_{i}}^{h}\right)_{i=1}^{M}\right\} \right)\overset{d}{=}f\left(\left\{ \left(\sigma_{r_{i}},B_{r_{i}}\right)_{i=1}^{M}\right\} \right)\text{.}\label{eqnchap12}
\end{equation}
In view of 
\begin{eqnarray*}
\int_{-\infty}^{t}\psi_{h}^{n}\left(t-s\right)dB_{s}^{h} & = & \sum\limits _{k\in\mathbb{Z}}\sigma_{\frac{k}{2^{n}}+h}\int_{-\infty}^{t}\phi^{n}\left(t-s\right)\mathbf{1}_{\left\{ \frac{k}{2^{n}}<s\leq\frac{k+1}{2^{n}}\right\} }dB_{s}^{h}\\
 & = & \sum\limits _{k\in\mathbb{Z}}\sum\limits _{j=1}^{m_{n}}b_{i}^{n}\sigma_{\frac{k}{2^{n}}+h}\int_{-\infty}^{t}\mathbf{1}_{\left\{ \frac{k}{2^{n}}\vee t-q_{i}<s\leq\frac{k+1}{2^{n}}\wedge t-p_{i}\right\} }dB_{s}^{h}\\
 & = & \sum\limits _{k\in\mathbb{Z}}\sum\limits _{j=1}^{m_{n}}b_{i}^{n}\sigma_{\frac{k}{2^{n}}+h}\left(B_{\left(\frac{k+1}{2^{n}}\wedge t-p_{i}\right)+h}-B_{\left(\frac{k}{2^{n}}\vee t-q_{i}\right)+h}\right),
\end{eqnarray*}
and (\ref{eqnchap12}), we have that for any $\left(c_{1},\ldots,c_{M}\right)$
\begin{eqnarray*}
\sum\limits _{j=1}^{M}c_{j}\int_{-\infty}^{r_{j}}\psi_{h}^{n}\left(r_{j}-s\right)dB_{s}^{h} & = & \sum\limits _{k\in\mathbb{Z}}\left(\sum\limits _{j=1}^{M}\sum\limits _{i=1}^{m_{n}}c_{j}b_{i}^{n}\right)\sigma_{\frac{k}{2^{n}}+h}\\
 &  & \times\left(B_{\left(\frac{k+1}{2^{n}}\wedge t-p_{i}\right)+h}-B_{\left(\frac{k}{2^{n}}\vee t-q_{i}\right)+h}\right)\\
 &  & \overset{d}{=}\sum\limits _{k\in\mathbb{Z}}\left(\sum\limits _{j=1}^{M}\sum\limits _{i=1}^{m_{n}}c_{j}b_{i}^{n}\right)\sigma_{\frac{k}{2^{n}}}\\
 &  & \times\left(B_{\frac{k+1}{2^{n}}\wedge t-p_{i}}-B_{\frac{k}{2^{n}}\vee t-q_{i}}\right)\\
 & = & \sum\limits _{j=1}^{M}c_{j}\int_{-\infty}^{r_{j}}\psi_{0}^{n}\left(r_{j}-s\right)dB_{s}\text{,}
\end{eqnarray*}
the desired result. \end{proof}

\bibliographystyle{chicago}
\bibliography{bibliographyBSSandrelated}

\end{document}